\documentclass[reqno]{amsart}
\usepackage{amscd}
\usepackage{amssymb}
\usepackage{amsmath}
\usepackage{amsthm}
\usepackage{latexsym}
\usepackage{upref}
\usepackage{epsfig}
\usepackage{mathrsfs}
\usepackage{float}
\usepackage[colorlinks,urlcolor=black]{hyperref}
\usepackage{graphics,graphicx,color}
\usepackage{floatflt}
\usepackage{indentfirst}
\usepackage[cmtip,arrow]{xy}
\usepackage{pb-diagram,pb-xy}
\usepackage{cite}

\newtheorem{thm}{Theorem}[section]

\newtheorem{cor}[thm]{Corollary}
\newtheorem{lem}[thm]{Lemma}
\theoremstyle{definition}

\newtheorem{exa}[thm]{Example}
\numberwithin{equation}{section}

\begin{document}
\title{On the bisymmetric nonnegative inverse eigenvalue problem}

\author{Somchai Somphotphisut}
\address{Department of Mathematics and Computer Science, Faculty of Science,
Chulalongkorn University, Phyathai Road, Patumwan, Bangkok 10330,
Thailand} \email{somchai.so@student.chula.ac.th \textrm{and
}kwiboonton@gmail.com}

\author{Keng Wiboonton}
\address{}
\email{}

\begin{abstract}
We study the \textit{bisymmetric nonnegative inverse eigenvalue problem} (BNIEP). This problem is the problem of finding the necessary and sufficient conditions on a list of $n$ complex numbers to be a spectrum of an $n \times n$ bisymmetric nonnegative matrix. Most recently, some of the sufficient conditions for the BNIEP are given by Julio and Soto in \cite{julio}. In this article, we give another proof of one result (Theorem 4.3) in \cite{julio} and we obtain the result very similar to the one (Theorem 4.2) in \cite{julio} using a different method to construct our desired bisymmetric nonnegative matrix. We also give some sufficient conditions for the BNIEP based on the sufficient conditions for the \textit{nonnegative inverse eigenvalue problem} (NIEP) given by Borobia in \cite{borobia}. We give the condition that is both necessary and sufficient for the BNIEP when $n \leq 4$ and then we show that for $n = 6$, the BNIEP and the \textit{symmetric nonnegative eigenvalue problem} (SNIEP) are different. Moreover, some sufficient conditions for the bisymmetric \textit{positive} inverse eigenvalue problem are provided. Finally, we give a new result on a sufficient condition for the BNIEP with the prescribed diagonal entries.
\end{abstract}
\maketitle

\section{Introduction}

The \textit{Nonnegative Inverse Eigenvalue Problem} was started when Kolmokorov \cite{Kol} asked the following question in 1937: When is a given complex number an eigenvalue of some nonnegative matrix? Later, in 1949, Suleimanova \cite{sul} extended this question to the problem of determining necessary and sufficient conditions for a list of $n$ complex numbers to be the eigenvalues of an $n \times n$ nonnegative matrix which nowadays is called the \textit{Nonnegative Inverse Eigenvalues Problem}(NIEP). The set of $n$ complex numbers is said to be \textit{realizable} if there is an $n \times n$ nonnegative matrix having these complex numbers as its eigenvalues. In particular, if we consider the list of $n$ real numbers then the problem is called the \textit{Real Nonnegative Inverse Eigenvalue Problem} (RNIEP). Sufficient conditions for the RNIEP have been studied by Suleimanova \cite{sul}, H. Perfect \cite{per}, P.G. Ciarlet \cite{ciarlet}, Kellogg \cite{Kel}, A. Brobia \cite{borobia}, Salzmann \cite{salz}, G.W. Soules \cite{soules}, G.Wuwen \cite{wuwen}, R.L. Soto and O.Rojo \cite{soto,soto2}. In this section, we collect several important sufficient conditions for the NIEP.

\bigskip
\begin{thm} (Suleimanova \cite{sul}, 1949) \label{suleimanova}

If $\lambda_0 \geq 0 \geq \lambda_1 \geq \ldots \geq \lambda_n$ are real numbers such that $\displaystyle \sum_{i=0}^{n} \lambda_i \geq 0$ then $\lbrace \lambda_0, \lambda_1, \ldots, \lambda_n \rbrace$ is realizable.
\end{thm}

\bigskip

\begin{thm} (Kellog \cite{Kel}, 1971) \label{kellogg}

If $\lambda_0 \geq \lambda_1 \geq \ldots \geq \lambda_M \geq 0 > \lambda_{M+1} \geq \lambda_{M+2} \geq \ldots \geq \lambda_n$ are real numbers and $\displaystyle K=\lbrace i \in \lbrace 1,2,\ldots, \emph{min}\lbrace M,n-M \rbrace \rbrace \vert \lambda_i + \lambda_{n-i+1} < 0 \rbrace$, and if

\qquad \qquad \qquad $\displaystyle \lambda_0 + \sum_{i \in K, i<k} (\lambda_i + \lambda_{n-i+1})+\lambda_{n-k+1} \geq 0$ \qquad for all $k \in K,$

\qquad \qquad \qquad $\displaystyle \lambda_0 + \sum_{i \in K} (\lambda_i + \lambda_{n-i+1})+ \sum_{j=M+1}^{n-M} \lambda_j \geq 0$
\\
then $\lbrace \lambda_0, \lambda_1, \ldots, \lambda_n \rbrace \text{ is realizable}.$
\end{thm}

\bigskip

\begin{thm} (Borobia \cite{borobia}, 1995) \label{borobia}

If $\lambda_0 \geq \lambda_1 \geq \ldots \geq \lambda_M \geq 0 > \lambda_{M+1} \geq \lambda_{M+2} \geq \ldots \geq \lambda_n$ are real numbers and there is a partition $\Lambda_1 \cup \ldots \cup \Lambda_S$ of $\lbrace \lambda_{M+1}, \ldots, \lambda_n \rbrace$ such that
$$\lambda_0 \geq \lambda_1 \geq \ldots \geq \lambda_M > \sum_{\lambda \in J_S} \lambda \geq \ldots \geq \sum_{\lambda \in J_1} \lambda$$
satisfies the Kellogg's condition, then $\lbrace \lambda_0, \lambda_1, \ldots, \lambda_n \rbrace$ is realizable.
\end{thm}

\bigskip

In \cite{fdl}, M. Fiedler showed that the sufficient conditions of Theorem \ref{suleimanova} and Theorem \ref{kellogg} are also the sufficient conditions for the existence of symmetric nonnegative matrices with the prescribed eigenvalues. The problem of finding the necessary and sufficient conditions on the list of $n$ complex numbers to be a spectrum of a symmetric nonnegative matrix is called the \textit{Symmetric Nonnegative Inverse Eigenvalue Problem} (SNIEP). In \cite{nizar}, N. Radwan improved the sufficient condition of Theorem \ref{borobia} to the sufficient condition for the SNIEP. Other sufficient conditions of the SNIEP were also studied in \cite{soto3, soto4, loewy, mcdonald, laffey}. The relations among sufficient conditions for the RNIEP (and also for the SNIEP) were collected and discussed in \cite{carlos}. Later on, the sufficient conditions for the NIEP were studied in certain specific classes of matrices, for example, for normal matrices, the problem was studied by N. Radwan \cite{nizar}, and for symmetric circulant matrices and symmetric centrosymmetric matrices (equivalent to bisymmetric matrices), the problem were investigated in \cite{rojo}. Most recently, in \cite{julio}, A.I. Julio and R.L.Soto presented some sufficient conditions for the persymmetric and bisymmetric NIEP. Moreover, they showed that the sufficient condition in Theorem \ref{suleimanova} is also the sufficient condition for the \textit{bisymmetric nonnegative inverse eigenvalue problem} (BNIEP).

Recall that a \textit{bisymmetric matrix} is a square matrix that is symmetric with respect to both the diagonal line from the upper-left to the lower-right and is also symmetric to the diagonal line from the lower-left to the upper-right. Any bisymmetric matrix is a symmetric and persymmetric matrix. Moreover, A squre matrix is a bisymmetric matrix if and only if it is a symmetric centrosymmetric matrix. Thus, a square matrix
$A$ is a bisymmetric matrix if and only if $A^T=A$ and $JAJ=A$, where $J$ is the reverse identity matrix, i.e., $J=\begin{pmatrix}
\textbf{0}&&1\\
 &\reflectbox{$\ddots$}&\\
1&&\textbf{0}
\end{pmatrix}.$
Many results on symmetric centrosymmetric (bisymmetric) matrices were discussed by P. Butler and A. Cantoni in \cite{can}. Here, we collect some of these results.

\bigskip

\begin{thm} (Cantoni and Butler \cite{can}, 1976) \label{canthm1}

Let $Q$ be an $n \times n$ bisymmetric matrix.

(1) If $n$ is an even number then $Q$ is of the form
$\begin{pmatrix}
A & JCJ \\
C & JAJ
\end{pmatrix}$, where $A$ and $C$ are $\frac{n}{2} \times \frac{n}{2}$ matrices $A=A^T$ and $C^T=JCJ.$

(2) If $n=2m+1$ is an odd number then $Q$ is of the form
$\begin{pmatrix}
A   & x  & JCJ \\
x^T & p  &x^TJ \\
C   & Jx &JAJ
\end{pmatrix}$, where $x$ is an $m \times 1$ matrix, $A$ and $C$ are $m \times m$ matrices, $A=A^T$ and $C^T=JCJ$.
\end{thm}

\bigskip

\begin{thm} (Cantoni and Butler \cite{can}, 1976) \label{canthm2}

Let $Q$ be an $n \times n$ bisymmetric matrix.

(1) If $n$ is an even number and $Q$ is of the form $Q=\begin{pmatrix}
A & JCJ \\
C & JAJ
\end{pmatrix}$
then $Q$ orthogonally similar to the matrix
$\begin{pmatrix}
A-JC &  \\
 & A+JC
\end{pmatrix}$.

(2) If $n=2m+1$ is an odd number and $Q$ is of the form
$Q=\begin{pmatrix}
A   & x  & JCJ \\
x^T & p  &x^TJ \\
C   & Jx &JAJ
\end{pmatrix}$
then $Q$ orthogonally similar to the matrix
$\begin{pmatrix}
A-JC &    &  \\
     &p   & \sqrt{2}x^T\\
     &\sqrt{2}x &A+JC
\end{pmatrix}$.
\end{thm}

\bigskip

Some sufficient conditions for the BNIEP are recently studied by Julio and Soto \cite{julio}. The origin of these sufficient conditions for the BNIEP in \cite{julio} comes from the rank-$r$ perturbation results due to Rado and introduced by Perfect. In \cite{soto2}, Soto and Rojo reapplied the results of Perfect into the sufficient conditions of NIEP. Moreover, in \cite{soto4} these sufficient conditions were used to be sufficient conditions for the SNIEP. Some related theorems about sufficient conditions of BNIEP are collected below:
\bigskip

\begin{thm}
(Rado \cite{rado}, 1955)

Let $A$ be an $n \times n$ arbitrary matrix with eigenvalues $\lambda_1, \lambda_2, \ldots, \lambda_n$ and, for some $r \leq n$, let $\Omega =\emph{diag}(\lambda_1, \lambda_2, \ldots, \lambda_r )$. Let $X$ be an $n \times r$ matrix with rank $r$ such that its columns $x_1, x_2, \ldots, x_r $ satisfy $Ax_i = \lambda_i x_i, i=1, 2, \ldots, r$. Let $C$ be an $r \times n$ arbitrary matrix. Then $A+XC$ has eigenvalues $\mu_1, \ldots, \mu_r, \lambda_{r+1}, \ldots, \lambda_{n}$, where $\mu_1, \ldots, \mu_r$ are eigenvalues of the matrix $\Omega+CX$.
\end{thm}

\bigskip

\begin{thm} (Soto \cite{soto4}, 2007) \label{sotos}

Let $A$ be an $n \times n$ symmetric matrix with eigenvalues $\lambda_1, \lambda_2, \ldots, \lambda_n$ and, for some $r \leq n$, let $\lbrace x_1, x_2, \ldots, x_r \rbrace$ be an orthonormal set of eigenvectors of $A$ spanning the invariant subspace associated with $\lambda_1, \lambda_2, \ldots, \lambda_r$. Let $X$ be the $n \times r$ matrix with the $i^{th}$ column $x_i$, let $\Omega = \emph{diag} (\lambda_1, \lambda_2, \ldots, \lambda_r)$, and let $C$ be any $r \times r$ symmetric matrix. Then the symmetric matrix $A+XCX^T$ has eigenvalues $\mu_1, \ldots, \mu_r, \lambda_{r+1}, \ldots, \lambda_{n}$, where $\mu_1, \ldots, \mu_r$ are eigenvalues of the matrix $\Omega+C$.
\end{thm}

\bigskip

\begin{thm} (Soto \cite{soto4}, 2007)\label{soto4}

Let $\lambda_1 \geq \lambda_1 \geq \ldots \geq \lambda_n$ be real numbers and, for some $t \leq n$, let $\omega_1, \ldots, \omega_t$ be real numbers satisfying $ 0 \leq \omega_k \leq \lambda_1, i=1, \ldots, t.$ If there exist

\qquad (1) a partition $\Lambda_1 \cup \ldots \cup \Lambda_t$ of $\lbrace \lambda_1, \ldots, \lambda_n \rbrace$, in which for each $j$, $ \Lambda_j=\lbrace \lambda_{j1}, \lambda_{j2}, \ldots,$ $\lambda_{jp_j} \rbrace$, $\lambda_{jk} \geq \lambda_{j(k+1)}$, $\lambda_{j1} \geq 0$ , and $\lambda_{11}=\lambda_1$, such that for each $j=1, \ldots,t$, the set $\Gamma_j=\lbrace \omega_j, \lambda_{j2},$ $\ldots, \lambda_{jp_{j}} \rbrace$ is realizable by a nonnegative symmetric matrix, and

\qquad (2) a $t \times t$ nonnegative symmetric matrix with all eigenvalues as $\lambda_{11}, \lambda_{21}, \ldots$, $\lambda_{t1}$ and diagonal entries as $\omega_1, \omega_2, \ldots, \omega_t,$ then\\
$\lbrace \lambda_1, \lambda_2, \ldots, \lambda_n \rbrace$ is realizable by nonnegative symmetric matrix.
\end{thm}

\bigskip

The next two results are the sufficient conditions for the BNEIP given by Julio and Soto in \cite{julio}.

\bigskip

\begin{thm} (A.I Julio and Soto \cite{julio}, 2015) \label{julio1}

Let $\Lambda = \lbrace \lambda_1, \lambda_2, \ldots, \lambda_n \rbrace$ be a list of real numbers with $\lambda_1 \geq \vert \lambda_i \vert, i= 2,3, \ldots, n,$ and $\displaystyle \sum_{i=1}^n \lambda_i \geq 0$. Suppose there exists a partition of $\Lambda$,

\qquad \qquad \qquad $\Lambda =\Lambda_0 \cup \Lambda_1 \cup \ldots \cup \Lambda_{\frac{p_0}{2}} \cup \Lambda_{\frac{p_0}{2}} \cup \ldots \cup \Lambda_1$, \.\ for even $\rho_0$,  with

\qquad \qquad \qquad $\Lambda_0=\lbrace \lambda_{01}, \lambda_{02}, \ldots, \lambda_{0p_0} \rbrace$, \qquad $\lambda_{01} = \lambda_1$,

\qquad \qquad \qquad $\Lambda_k=\lbrace \lambda_{k1}, \lambda_{k2}, \ldots, \lambda_{kp_k} \rbrace$, \qquad $k=1, 2, \ldots, \frac{p_0}{2}$, \\
where some of the lists $\Lambda_k$ can be empty, such that the following conditions are satisfied:

\qquad (1) For each $k=1, 2, \ldots, \frac{p_0}{2}$, there exists a symmetric matrix with all eigenvalues as $\omega_k, \lambda_{k1}, \lambda_{k2}, \ldots, \lambda_{kp_k}$, $0\leq \omega_k \leq \lambda_1$.

\qquad (2) There exists a bisymmetric nonnegative matrix of order $p_0$, with all eigenvalues $\lambda_{01}, \lambda_{02}, \ldots, \lambda_{0p_0}$ and diagonal entries $\omega_1, \omega_2, \ldots, \omega_{\frac{p_0}{2}},\omega_{\frac{p_0}{2}}, \ldots, \omega_2, \omega_1$. \\
Then $\lbrace \lambda_1, \lambda_2, \ldots, \lambda_n \rbrace$ is realizable by an $n \times n$ nonnegative bisymmetric matrix.
\end{thm}

\bigskip

\begin{thm} (A.I Julio and Soto \cite{julio}, 2015) \label{julio2} 

Let $\Lambda = \lbrace \lambda_1, \lambda_2, \ldots, \lambda_n \rbrace$ be a list of real numbers with $\lambda_1 \geq \vert \lambda_i \vert, i= 2,3, \ldots, n,$ and $\displaystyle \sum_{i=1}^n \lambda_i \geq 0$. Suppose there exists a partition of $\Lambda$,

\qquad \qquad $\Lambda =\Lambda_0 \cup \Lambda_1 \cup \ldots \cup \Lambda_{\frac{p_0-1}{2}} \cup \Lambda_{\frac{p_0+1}{2}} \cup \Lambda_{\frac{p_0-1}{2}} \cup \ldots \cup \Lambda_1$, \.\ for odd $\rho_0$,  with

\qquad \qquad $\Lambda_0=\lbrace \lambda_{01}, \lambda_{02}, \ldots, \lambda_{0p_0} \rbrace$, \qquad $\lambda_{01} = \lambda_1$,

\qquad \qquad $\Lambda_k=\lbrace \lambda_{k1}, \lambda_{k2}, \ldots, \lambda_{kp_k} \rbrace$, \qquad $k=1, 2, \ldots, \frac{p_0+1}{2}$, \\
where some of the lists $\Lambda_k$ can be empty, such that the following conditions are satisfied:

(1) For each $k=1, 2, \ldots, \frac{p_0-1}{2}$, there exists a symmetric matrix with all eigenvalues as $\omega_k, \lambda_{k1}, \lambda_{k2}, \ldots, \lambda_{kp_k}$, $0\leq \omega_k \leq \lambda_1$, and there exists a bisymmetric nonnegative matrix with all eigenvalues as $\omega_{d}, \lambda_{d1}, \lambda_{d2}, \ldots, \lambda_{dp_d}$, where $d = \frac{p_0+1}{2}.$

(2) There exists a $p_0 \times p_0$ bisymmetric nonnegative matrix of order $p_0$, with all eigenvalues $\lambda_{01}, \lambda_{02}, \ldots, \lambda_{0p_0}$ and diagonal entries $\omega_1, \omega_2,$ $\ldots, \omega_{\frac{p_0-1}{2}}, \omega_{\frac{p_0+1}{2}}, \omega_{\frac{p_0-1}{2}},$ $\ldots, \omega_2, \omega_1$. \\
Then $\lbrace \lambda_1, \lambda_2, \ldots, \lambda_n \rbrace$ is realizable by an $n \times n$ nonnegative bisymmetric matrix.
\end{thm}

\bigskip

In this article, we give some sufficient conditions for the BNIEP and also obtain some related results. In Section 2, We first start with the necessary and sufficient conditions for BNIEP in the case $n \leq 4$ and then we show that SNIEP and BNIEP are different when $n = 6$. In Section 3, we provide the sufficient conditions for the BNIEP based on the sufficient conditions of the NIEP and the SNIEP. Also, some results about the BNIEP due to Julio and Soto \cite{julio} are included and proven again with a different method. In Section 4, we extend the results of Section 3 to the results on the bisymmetric positive matrices. Then we provide the sufficient conditions for the BNEIP with the prescribed diagonal entries in the last section.

\bigskip

\section{The BNIEP for Matrices Having Small Sizes}

Let $\lambda_1 \geq \lambda_2 \geq \ldots \geq \lambda_n$ be real numbers. The well-known necessary conditions for an existence of an $n \times n$ nonnegative matrix with all eigenvalues as $\lambda_1, \lambda_2, \ldots, \lambda_n$ are $\displaystyle \sum_{i=1}^n \lambda_i \geq 0$ and $\lambda_1 \geq |\lambda_n|.$ Moreover, if $n \leq 4$ then these necessary conditions are also sufficient conditions for the RNIEP. In 1997 Wuwen \cite{wuwen} showed that the RNIEP and the SNIEP are equivalent for any list of $n \leq 4$ real numbers. In fact, the fact that the RNIEP and the SNIEP are different was proven by Johnson et al. in \cite{johnson}. In this section we will show that, for $n \leq 4$, those necessary conditions for the RNIEP are also sufficient conditions for an existence of an $n \times n$ nonnegative bisymmetric matrix with all eigenvalues as $\lambda_1, \lambda_2, \ldots, \lambda_n$. Also, we show that the SNIEP is different from the BNIEP when $n=6$. We start this section by the following lemmas.

\medskip

\begin{lem} \label{canlemma1}
Let $Q$ be an $n$ $\times$ $n$ nonnegative bisymmetric matrix with the Perron root $\lambda_0.$

(a) If $n$ is even and $Q$ is of the form $Q=\begin{pmatrix}
A & JCJ \\ C & JAJ
\end{pmatrix}$ then $\lambda_0$ is the Perron root of $A+JC.$

(b) If $n=2m+1$ is odd and $Q$ is of the form
$\begin{pmatrix} A & x & JCJ \\ x^{T} & q & x^{T}J \\ C & Jx & JAJ \end{pmatrix}$
then $\lambda_0$ is the Perron root of $\begin{pmatrix}
p & \sqrt{2}x^{T} \\ \sqrt{2}x & A+JC \end{pmatrix}$.
\begin{proof}
Since $Q$ is a nonnegative matrix with the Perron root $\lambda_0,$ $\lambda_0$ is also an eigenvalue of $Q$.

(a) Let $v=\begin{pmatrix}
v_1 \\ v_2
\end{pmatrix}$ be an eigenvector of $Q$ corresponding to $\lambda_0.$ Then $Qv=\lambda_0v$ implies $$\begin{pmatrix}
A & JCJ \\ C & JAJ
\end{pmatrix}
\begin{pmatrix}
v_1\\v_2
\end{pmatrix}=\lambda_0 \begin{pmatrix}
v_1 \\ v_2
\end{pmatrix}.$$
So,

\qquad \qquad \qquad $Av_1+JCJv_2$ = $\lambda_0v_1$ and
$Cv_1+JAJv_2$ = $\lambda_0v_2$. \\
Then
$$(A+JC)\left(\frac{v_1+Jv_2}{2}\right)=\frac{Av_1+JCv_1}{2}+\frac{AJv_2+JCJv_2}{2}=\lambda_0\left(\frac{v_1+Jv_2}{2}\right).$$
This shows that $\lambda_0$ is an eigenvalue of $A+JC$. Since all eigenvalues of $A+JC$ are eigenvalues of $Q$ and $\lambda_0$ is the Perron root of $Q$, $\lambda_0$ is also the Perron root of $A+JC$.

(b) Let $v=\begin{pmatrix}
v_1 \\ c \\ v_2
\end{pmatrix}$ be an eigenvector of $Q$ corresponding to $\lambda_0.$ By the argument as in part (a), we can show that $\lambda_0$ is the Perron root of
$\begin{pmatrix}
q & \sqrt{2}x^{T} \\ \sqrt{2}x & A+JC
\end{pmatrix}$ with a corresponding eigenvector $\begin{pmatrix}
\sqrt{2}c \\ v_1+Jv_2
\end{pmatrix}.$
\end{proof}
\end{lem}

\bigskip

\begin{lem} \label{canlemma2}
Let $Q$ be an $n \times n$ nonnegative bisymmetric matrix. If $v$ is an eigenvector corresponding to an eigenvalue $\lambda$ of $Q$ then so is $Jv$. Moreover, if $\lambda_0$ is the Perron root of $Q$ then there is a nonnegative eigenvector $v_0$ such that $Jv_0=v_0$.
\begin{proof}
Let $v$ be an eigenvector of $Q$ corresponding to the Perron root $\lambda$. So, $Qv=\lambda v$. Since $JQJ=Q$, $JQ=QJ$, and hence $QJv=JQv=\lambda Jv$. This shows that $Jv$ is also an eigenvector of $Q$ corresponding to $\lambda$. Now, let $\begin{pmatrix}
v_1\\
v_2
\end{pmatrix}$
be a nonnegative eigenvector corresponding to $\lambda_0$ where $\lambda_0$ is the Perron root of $Q$. If $n$ is even number, by Theorem \ref{canthm1}, $Q$ is of the form $Q=\begin{pmatrix}
A & JCJ \\ C & JAJ
\end{pmatrix}$, and thus we have
$Q\begin{pmatrix}
v_1\\v_2
\end{pmatrix}=\begin{pmatrix}
A & JCJ \\ C & JAJ
\end{pmatrix} \begin{pmatrix}
v_1\\v_2
\end{pmatrix}=\lambda_0 \begin{pmatrix}
v_1\\v_2
\end{pmatrix}$.
It follows that

$$Av_1+JCJv_2 = \lambda_0v_1,  \text{ and } Cv_1+JAJv_2 = \lambda_0v_2.$$ So,
$$\begin{pmatrix}
A & JCJ \\ C & JAJ
\end{pmatrix}
\begin{pmatrix}
\frac{v_1+Jv_2}{2} \\
\frac{Jv_1+v_2}{2}
\end{pmatrix}=\begin{pmatrix}
\frac{Av_1+AJv_2+JCv_1+JCJv_2}{2}\\
\frac{Cv_1+CJv_2+JAJv_2+JAv_1}{2}
\end{pmatrix}=\lambda_0 \begin{pmatrix}
\frac{v_1+Jv_2}{2} \\
\frac{Jv_1+v_2}{2}
\end{pmatrix}.$$
This shows that $\dfrac{1}{2}\begin{pmatrix}
v_1+Jv_2\\
Jv_1+v_2
\end{pmatrix}$ is a nonnegative eigenvector corresponding to $\lambda_0$ as desire.

On the other hand, if $n=2m+1$ is an odd number and $\begin{pmatrix}
v_1\\c\\v_2
\end{pmatrix}$ is a nonnegative eigenvector corresponding to the Perron root $\lambda_0$ then we can also show that the nonnegative vector $\dfrac{1}{2}\begin{pmatrix}
v_1+Jv_2\\2c\\Jv_1+v_2
\end{pmatrix}$ is an eigenvector corresponding to $\lambda_0$.
\end{proof}
\end{lem}

\bigskip

\begin{lem} \label{lep} 
Let $A$ be an $m \times m$ nonnegative bisymmetric matrix with eigenvalues $\alpha_1 \geq \alpha_2 \geq \ldots \geq \alpha_m$ and let $u_1$ be a unit nonnegative eigenvector corresponding to $ \alpha_1 $ such that $Ju_1=u_1$; let $B$ be an $n \times n$ nonnegative bisymmetric matrix with eigenvalues $\beta_1 \geq \beta_2 \geq \ldots \geq \beta_n$ and let $v_1$ be a unit nonnegative eigenvector corresponding to $ \beta_1 $ such that $Jv_1=v_1$. If $\gamma_1 , \gamma_2 , \gamma_3$ are all eigenvalues of the matrix
$$\widehat{C} = \begin{pmatrix}
\beta_1 & \rho & \xi \\ \rho & \alpha_1 & \rho \\ \xi & \rho & \beta_1
\end{pmatrix},$$
where $\rho$, $\xi \geq$ 0, then the matrix
$$C = \begin{pmatrix}
B&\rho v_1 u_1^{T}&\xi v_1 v_1^{T} \\ \rho u_1 v_1^{T}&A& \rho u_1 v_1^{T} \\ \xi v_1 v_1^{T}&\rho v_1 u_1^{T}&B
\end{pmatrix}$$
is a nonnegative bisymmetric matrix with all eigenvalues as
$$\gamma_1, \gamma_2, \gamma_3, \alpha_2, \alpha_3, \ldots, \alpha_m, \beta_2, \beta_3, \ldots, \beta_n, \beta_2, \beta_3 , \ldots, \beta_n.$$
\begin{proof}
Let $u_1, u_2, \ldots, u_m$ be an orthonormal system of eigenvectors of $A$. Similarly, let $v_1$, $v_2$, $\ldots$, $v_n$ form an orthonormal system of eigenvectors of $B$ and $\begin{pmatrix} r_1 \\ s_1 \\ t_1 \end{pmatrix}, \begin{pmatrix} r_2 \\ s_2 \\ t_2 \end{pmatrix},\begin{pmatrix} r_3 \\ s_3 \\ t_3 \end{pmatrix}$ form an orthonormal system of eigenvectors of $\widehat{C}$. It's easy to see that  $\begin{pmatrix} 0 \\ u_i \\ 0 \end{pmatrix}, \begin{pmatrix} v_j \\ 0 \\ 0 \end{pmatrix}, \begin{pmatrix} 0 \\ 0 \\ v_j \end{pmatrix}, \begin{pmatrix} r_k v_1 \\ s_k u_1 \\ t_k v_1 \end{pmatrix}$ are m+2n linearly independent eigenvectors of C corresponding to eigenvalues $\alpha_i , \beta_j , \gamma_k$ for $i=2,\ldots, m, j=2,\ldots, n,$ and $k=1,2,3$, respectively.

We see that $C$ is a nonnegative matrix since $A$ and $B$ are nonnegative symmetric matrices with Perron roots $\alpha_1$ and $\beta_1$, respectively, $u_1$ and $v_1$ are nonnegative vectors and $\rho, \xi \geq 0$. Moreover, $C$ is bisymmetric because $A$ and $B$ are bisymmetric matrices and $Ju_1=u_1, Jv_1=v_1.$
\end{proof}
\end{lem}
\bigskip

Note that Lemma \ref{lep} is a special case of Theorem \ref{julio2}. However, we need this lemma in order to construct our desired bisymmetric nonnegative matrix. In Lemma \ref{lep}, we see that the eigenvalues $\gamma_1$, $\gamma_2$ and $\gamma_3$ of $\widehat{C}$ depend on $\rho$ and $\xi$. If $\alpha_1 \geq \beta_1, \rho = \sqrt{\frac{-(\alpha_1-\beta_1-a)(a+b)}{2}}$ and $\xi = -b$, where $a$, $b$ are real numbers such that $\alpha_1-\beta_1$ $\geq$ $a$ $\geq b$ and $a+b \leq 0$, then $\alpha_1 - (a+b)$, $\beta_1+a$ and $\beta_1+b$ are all eigenvalues of $\widehat{C}$. Then we have the following result.

\bigskip

\begin{cor} \label{cor1}
Let $A$ and $B$ be nonnegative bisymmetric matrices as in Lemma \ref{lep} with $\alpha_1 \geq \beta_1.$ If $a$, $b$ are real numbers such that $\alpha_1-\beta_1 \geq a \geq b$ and $a+b \leq 0$ then there is a nonnegative bisymmetric matrix with all eigenvalues as $\alpha_1 - (a+b), \beta_1+a$, $\beta_1+b$, $\alpha_2$ , $\alpha_3$, $\ldots$, $\alpha_m$ , $\beta_2$, $\beta_3$, $\ldots$, $\beta_n$, $\beta_2$, $\beta_3$, $\ldots$, $\beta_n$.
\end{cor}

\bigskip
Let $A$ and $B$ be square matrices and $B$ is of the form $\begin{pmatrix} B_{11} & B_{12} \\ B_{21} & B_{22} \end{pmatrix}$. We know that the eigenvalues of matrix $\begin{pmatrix} B_{11} & & B_{12} \\ &A& \\ B_{21} & & B_{22} \end{pmatrix}$ are all the eigenvalues of the matrix $\begin{pmatrix} A & \\ & B \end{pmatrix}$. We use this fact and Lemma \ref{lep} to construct a nonnegative bisymmetric matrix with all eigenvalues $\lambda_1 \geq \lambda_2 \geq \ldots \geq \lambda_n$, where $n \leq 4$, $\lambda_1 \geq |\lambda_n |$ and $\sum_{i=1}^n \lambda_{i} \geq 0$.

\bigskip

\begin{thm} \label{small}
Let $\lambda_1 \geq \lambda_2 \geq \ldots \geq \lambda_n$ be real numbers. If $n \leq 4$ then the necessary conditions $\lambda_1 \geq |\lambda_n |$ and $\displaystyle \sum_{i=1}^n \lambda_{i} \geq 0$ are also the sufficient conditions for an existence of nonnegative bisymmetric matrix with all eigenvalues as $\lambda_1, \lambda_2, \ldots, \lambda_n.$
\begin{proof}

First, we consider when $n=2$. In this case, we have $\lambda_1 \geq \lambda_2$ and $\lambda_1 + \lambda_2 \geq 0$. Then the matrix $Q = \begin{pmatrix} \frac{\lambda_1+\lambda_2}{2} & \frac{\lambda_1-\lambda_2}{2} \\ \frac{\lambda_1-\lambda_2}{2} & \frac{\lambda_1+\lambda_2}{2} \end{pmatrix}$ is a nonnegative bisymmetric matrix with eigenvalues $\lambda_1$ and $\lambda_2$.

Next, suppose that $n=3$. In this case, we have $\lambda_1 + \lambda_3 \geq 0$.
If $\lambda_2 \geq 0$, then the matrix $Q = \begin{pmatrix} \frac{\lambda_1+\lambda_3}{2} & &\frac{\lambda_1-\lambda_3}{2} \\ & \lambda_2 & \\ \frac{\lambda_1-\lambda_3}{2} & &\frac{\lambda_1+\lambda_3}{2} \end{pmatrix}$ will be a desired nonnegative bisymmetric matrix. If $\lambda_2 < 0$, then using Corollary \ref{cor1} with $A=\begin{pmatrix}\lambda_1+\lambda_2+\lambda_3  \end{pmatrix}$, $B=\begin{pmatrix}
0 \end{pmatrix}$, $a=\lambda_2$ and $b=\lambda_3$, we can construct the desired matrix $Q$ as
$$Q =\begin{pmatrix} 0 & \sqrt{\frac{-(\lambda_1+\lambda_3)(\lambda_2+\lambda_3)}{2}} & -\lambda_3 \\ \sqrt{\frac{-(\lambda_1+\lambda_3)(\lambda_2+\lambda_3)}{2}} & \lambda_1+\lambda_2+\lambda_3 & \sqrt{\frac{-(\lambda_1+\lambda_3)(\lambda_2+\lambda_3)}{2}} \\ -\lambda_3 &\sqrt{\frac{-(\lambda_1+\lambda_3)(\lambda_2+\lambda_3)}{2}}&0 \end{pmatrix}.$$

Finally, we consider the case $n=4$. In this case, we have $\lambda_1 + \lambda_4 \geq 0$. If $\lambda_2 + \lambda_3 \geq 0$ then the matrix $A=\begin{pmatrix} \frac{\lambda_2+\lambda_3}{2} & \frac{\lambda_2-\lambda_3}{2} \\ \frac{\lambda_2-\lambda_3}{2} & \frac{\lambda_2+\lambda_3}{2} \end{pmatrix}$ is a nonnegative matrix with eigenvalues $\lambda_2, \lambda_3$ and the matrix $B=\begin{pmatrix} \frac{\lambda_1+\lambda_4}{2} & \frac{\lambda_1-\lambda_4}{2} \\ \frac{\lambda_1-\lambda_4}{2} & \frac{\lambda_1+\lambda_4}{2} \end{pmatrix}$ is a nonnegative matrix with eigenvalues $\lambda_1, \lambda_4$. Then the nonnegative bisymmetric matrix
$$Q= \begin{pmatrix} \frac{\lambda_1+\lambda_4}{2} & &\frac{\lambda_1-\lambda_4}{2} \\ & A & \\ \frac{\lambda_1-\lambda_4}{2} & & \frac{\lambda_1+\lambda_4}{2} \end{pmatrix}$$
is the desired matrix. If $\lambda_2 + \lambda_3 < 0$ then using Corollary \ref{cor1} with

$A=\begin{pmatrix} \frac{\lambda_1+\lambda_2+\lambda_3+\lambda_4}{2} & \frac{\lambda_1+\lambda_2+\lambda_3-\lambda_4}{2} \\ \frac{\lambda_1+\lambda_2+\lambda_3-\lambda_4}{2} & \frac{\lambda_1+\lambda_2+\lambda_3+\lambda_4}{2} \end{pmatrix}$
, $B=\left(0\right)$, $a=\lambda_2$ and
$b=\lambda_3$,

we can construct the desired matrix $Q$ as
$$\begin{pmatrix}
0 & \sqrt{\frac{-(\lambda_1+\lambda_3)(\lambda_2+\lambda_3)}{2}}  u^{T}_1 & -\lambda_3 \\ \sqrt{\frac{-(\lambda_1+\lambda_3)(\lambda_2+\lambda_3)}{2}}u_1 & A & \sqrt{\frac{-(\lambda_1+\lambda_3)(\lambda_2+\lambda_3)}{2}}u_1 \\ -\lambda_3 &\sqrt{\frac{-(\lambda_1+\lambda_3)(\lambda_2+\lambda_3)}{2}} u^{T}_1 & 0
\end{pmatrix},$$
where $u_1= \begin{pmatrix} \frac{1}{\sqrt{2}} \\ \frac{1}{\sqrt{2}} \end{pmatrix}$
is a unit eigenvector of A corresponding to eigenvalue $\lambda_1+\lambda_2+\lambda_3$.
\end{proof}
\end{thm}
\bigskip

By Theorem \ref{small}, the BNIEP and the SNIEP are equivalent for $n \leq 4$. Next, we show that there is no $6 \times 6$ bisymmetric nonnegative matrix with the eigenvalues as $6, 6, -2, -3, -3, -4$ while the symmetric nonnegative matrix $$\begin{pmatrix}
0&3&3&&&\\
3&0&3&&&\\
3&3&0&&&\\
&&&0&\sqrt{6}&4\\
&&&\sqrt{6}&0&\sqrt{6} \\
&&&4&\sqrt{6}&0
\end{pmatrix}$$ has them as its eigenvalues, that is, the BNIEP and the SNIEP are different for $n=6$.
\bigskip

\begin{thm}
For $n=6$, the BNIEP and the SNIEP are different.
\begin{proof}
It suffices to show that there is no $6 \times 6$ bisymmetric nonnegative matrix with eigenvalues as $6, 6, -2, -3, -3, -4$. Suppose $Q$ is a $6 \times 6$ nonnegative bisymmetric matrix with these eigenvalues. Then $Q$ must be a reducible matrix by Perron-Frobenius Theorem. Therefore there is a permutation matrix $P$ such that
$$P^TQP=\begin{pmatrix}
A_1 & \\
& A_2
\end{pmatrix}:=S,$$
where $A_1, A_2$ are $3 \times 3$ nonnegative symmetric matrices with the set of all eigenvalues as $\lbrace 6, -3, -3 \rbrace$ and $\lbrace 6, -2, -4 \rbrace$, respectively. Since the matrix $\begin{pmatrix}
0&3&3\\
3&0&3\\
3&3&0
\end{pmatrix}$ is the only $3 \times 3$ nonnegative symmetric matrix which has eigenvalues $6, -3, -3$, $A_1$ must be this matrix. Let $A_2=\begin{pmatrix}
0&\alpha &\beta \\
\alpha &0&\gamma \\
\beta &\gamma &0
\end{pmatrix}$ where $\alpha, \beta, \gamma \geq 0.$ Then we have
$$\begin{pmatrix}
0 & a & b & c & d & e\\
a & 0 & f & g & h & d\\
b & f & 0 & i & g & c\\
c & g & i & 0 & f & b\\
d & h & g & f & 0 & a\\
e & d & c & b & a & 0
\end{pmatrix} =: Q = P\begin{pmatrix}
0 & 3 & 3 &        &        &  \\
3 & 0 & 3 &        &        &  \\
3 & 3 & 0 &        &        &  \\
  &   &   & 0      & \alpha & \beta\\
  &   &   & \alpha & 0      & \gamma\\
  &   &   & \beta  & \gamma & 0
\end{pmatrix}P^T := PSP^T. \ \ \ldots (\ast) $$
Since $6, -2, -4$ are all eigenvalues of $A_2$, $\alpha, \beta$ and $\gamma$ can not be identical. If one of $\alpha, \beta$ and $\gamma$ is $0$, then $0$ is an eigenvlues of $A_2$, which is impossible. Therefore $\alpha, \beta$ and $\gamma$ are positive numbers. Now, we consider in two cases.

\textit{Case 1}: $\alpha, \beta$ and $\gamma$ are all distinct positive numbers.
In this case, at least two numbers of $\alpha, \beta$ and $\gamma$ must be distinct from $3$. WLOG, let $\alpha\neq 3, \beta \neq 3$ and $\alpha \neq \beta$.

Since each of value $\alpha$ and $\beta$ appears twice in the matrix $S$, $\alpha$ and $\beta$ must be on the secondary main diagonal line of the matrix $Q$. But $\alpha$ and $\beta$ are in the forth column of $S$. This is impossible because any two entries from the same column of $S$ cannot be permuted by the permutation matrix $P$ to two entries lying in the different columns of $Q$ .

\textit{Case 2}: Two of the numbers $\alpha, \beta$ and $\gamma$ are same positive numbers and different from another one. Without loss of generality, let $\alpha=\beta$ and $\alpha \neq \gamma$.

If $\alpha \neq 3$ then $\alpha$ appears four times in the matrix $S.$ Then $\alpha $ and $\beta$ must be on the different columns of $Q.$ This is impossible by the similar reasoning as in case 1.

Finally, if $\alpha = \beta=3$ and $\gamma \neq 3,$ then the characteristic polynomial of $A_2$ is $x^3-(\gamma^2+18)x-18\gamma$ which can not be equal to $(x-6)(x+2)(x+4)$ for any positive number $\gamma$.

This shows that there is no $6 \times 6$ bisymmetric nonnegative matrix with eigenvalues as $6, 6, -2, -3, -3, -4.$
\end{proof}
\end{thm}

\section{Sufficient Conditions for the BNIEP}
In this section, we will find sufficient conditions for the existence of a nonnegative bisymmetric matrix and a positive bisymmetric matrix with prescribed eigenvalues. We begin this section by considering the sufficient condition of Theorem \ref{suleimanova} in which it was improved to the bisymmetric case by A. Julio and Soto \cite{julio}. However, we give another proof using Lemma \ref{lep}.

\begin{thm} (Julio and Soto, Theorem 4.3 in\cite{julio}) \label{impsul}

Let $\lambda_0 \geq 0 \geq \lambda_1 \geq \ldots \geq \lambda_n$ be real numbers. If $\displaystyle \sum_{i=0}^n \lambda_{i} \geq 0$  then there is an $(n+1)\times (n+1)$ nonnegative bisymmetric matrix with all eigenvalues as $\lambda_0 , \lambda_1 , \ldots, \lambda_n.$
\begin{proof}
We will prove this theorem by induction on $n$. If $n=0$, the assertion is clear. If $n$ = 1,2 and 3, the assertions follow from Theorem \ref{small}.

Let $n \geq 4$ and suppose that the assertion is true for all smaller systems of numbers such that satisfy our assumption.
If $\lambda_0$ = 0 then $\lambda_i = 0$, for all $i=0, 1, 2,...,n,$ and the zero matrix of size $(n+1)\times (n+1)$ is the required matrix.
Now suppose $\lambda_0 > 0.$ Clearly, the system $\lambda_0 + \lambda_1 +\lambda_2, \lambda_3, \lambda_4,..., \lambda_n$ satisfies the assumption. By the induction hypothesis, there is a nonnegative bisymmetric matrix $A$ with eigenvalues $\lambda_0 + \lambda_1 + \lambda_2, \lambda_3, \lambda_4,..., \lambda_n$. It's easy to check that the matrix $A$ and $B=(0)$, $a= \lambda_1$ and $b=\lambda_2$ satisfy the condition in Corollary \ref{cor1}. Therefore, there is a nonnegative bisymmetric matrix $Q$ with all eigenvalues as $\lambda_0, \lambda_1,..., \lambda_n$.
\end{proof}
\end{thm}
\bigskip

Now, we consider the sufficient condition in theorem \ref{borobia}. In section 2, we proved that there is no $6 \times 6$ bisymmetric nonnegative matrix with eigenvalues $6, 6, -2, -3, -3, -4$ even though this list of numbers satisfies the condition of Theorem \ref{borobia}. This implies that the sufficient condition of Theorem \ref{borobia} can not be improved to the bisymmetric case. However, if the size of each patition $\Lambda_k$ of $\Lambda$ in Theorem \ref{borobia} is odd for all $k \leq \text{min} \lbrace M, S \rbrace,$ then we show that with this additional assumption together with the assumption in Theorem \ref{borobia}, there is a nonnegative bisymmetric matrix with the prescribed eigenvalues. We start with the following lemma.
\bigskip

\begin{lem} \label{lf3}
Let $Q_1$ and $Q_2$ be nonnegative bisymmetric matrices with all eigenvalues $\alpha_0 \geq \alpha_1 \geq \ldots \geq \alpha_m$ and $\beta_0 \geq \beta_1 \geq \ldots \geq \beta_n.$ If $\alpha_0 \geq \beta_0$ and $m$ or $n$ is odd, then for any $\varepsilon \geq 0$, there is a nonnegative bisymmetric matrix with all eigenvalues as $\alpha_0+\varepsilon, \beta_0-\varepsilon, \alpha_1 , \ldots, \alpha_m,\beta_1, \ldots, \beta_n.$
\begin{proof}
Suppose m is odd and n is even. Then $Q_1=\begin{pmatrix}
A & JCJ \\ C & JAJ
\end{pmatrix}$ and $Q_2=\begin{pmatrix}
D   & y  & JEJ \\
y^T & q  & y^TJ \\
E   & Jy & JDJ
\end{pmatrix},$
where $A, C$ are $(\frac{m+1}{2})$ $\times$ $(\frac{m+1}{2})$ matrices, $D, E$ are $\frac{n}{2} \times \frac{n}{2}$ matrices, and $A=A^T, D=D^T, C^T=JCJ, E^T=JEJ.$
Now, let $\alpha_0,\alpha_{1_1},\alpha_{1_2},...,\alpha_{1_\frac{m+1}{2}}$ be all eigenvalues of the matrix $A+JC$ and $\beta_0, \beta_{1_1}, \beta_{1_2},..., \beta_{1_\frac{n}{2}}$ be all eigenvalues of the matrix $\begin{pmatrix}
q & \sqrt{2}y^T \\ \sqrt{2}y & A+JC
\end{pmatrix}.$
By Theorem 2.1 and Theorem 2.2 in \cite{fdl}, we have
the matrix
$$\begin{pmatrix}
A+JC                & \rho cu_0      & \rho u_0 v_0^T    \\
\rho cu_0^T         &      q         & \sqrt{2}y^T       \\
\rho v_0 u_{0}^{T}  & \sqrt{2}y      &  D+JE
\end{pmatrix},$$
where $\rho = \sqrt{\varepsilon(\alpha_0-\beta_0+\varepsilon)}$, and $u_0$ and $\begin{pmatrix} c \\ v_0 \end{pmatrix}$
are nonnegative unit eigenvectors corresponding to $\alpha_0$, $\beta_0$, respectively, has all eigenvalues as $\alpha_0+\varepsilon$, $\beta_0-\varepsilon$, $\alpha_{1_1}$, $\alpha_{1_2}$, $\ldots$,$\alpha_{1_\frac{m+1}{2}},\beta_{1_1}$, $\beta_{1_2}$, $\ldots$, $\beta_{1_\frac{n}{2}}$, and this matrix is similar to the matrix
$\begin{pmatrix}
A+JC                & \rho u_0 v_0^T & \rho cu_0         \\
\rho v_0 u_{0}^{T}  &  D+JE          & \sqrt{2}y         \\
\rho cu_0^T         & \sqrt{2}y^T    &      q              \
\end{pmatrix}.$
Finally, let the matrix
$$X+JY = \begin{pmatrix}
A+JC           & \rho u_0 v_0^T \\
\rho v_0 u_0^T & D+JE
\end{pmatrix},$$
$$X-JY=\begin{pmatrix}
D-JE &     \\
     & A-JC
\end{pmatrix},$$ and
$w=\frac{1}{\sqrt{2}}\begin{pmatrix}
\rho cu_0 \\ \sqrt{2}y
\end{pmatrix}.$
Then the matrix
$$\begin{pmatrix}
X  & w & JYJ \\
w^T& q & w^T J \\
Y  & Jw & JXJ
\end{pmatrix}$$ is the desired matrix.
In the other cases, we can construct our nonnegative bisymmetric matrices in a similar way as we constructed above.
\end{proof}
\end{lem}

\bigskip

\begin{thm} \label{impboro}
Let $\lambda_0 \geq \lambda_1 \geq \ldots \geq \lambda_M \geq 0 > \lambda_{M+1} \geq \ldots \geq \lambda_n$ be real numbers. If there exists a partition $\Lambda =\Lambda_1 \cup \Lambda_2 \cup \cdots \cup \Lambda_S$ of $\lbrace\lambda_{M+1},\lambda_{M+2}, \ldots,\lambda_n\rbrace$, with $M \leq S$, $|\Lambda_j|$ is odd for $j=1,2,\ldots,M,$ and $T_k =$ $\displaystyle \sum_{\lambda_i \in \Lambda_k} \lambda_{i}$, with $T_S \geq T_{S-1} \geq \ldots \geq T_1$ satisfying

\qquad (1) $\displaystyle \lambda_0 + \sum_{i\in K, i<k} (\lambda_i + T_i)+ T_k \geq 0$, for all $k \in K$, and

\qquad (2) $\displaystyle \lambda_0 + \sum_{i\in K} (\lambda_i + T_i)+ \sum_{j=M+1}^S T_j  \geq 0,$
\\where $K=\lbrace i \in \lbrace 1,2,\ldots,M\rbrace \mid \lambda_i + T_i < 0\rbrace$, then there is a nonnegative bisymmetric matrix with all eigenvalues as $\lambda_0,\lambda_1,...,\lambda_n.$

\begin{proof}
We prove by induction on $M$. If $M=0$ and $S \geq 1$, the result follows from Theorem \ref{impsul}.

If $M=1, S=1$ and 1 $\notin K$, then $\lambda_1 + T_1 \geq 0$ implies $\lambda_0 + T_1 \geq 0.$ Thus by Theorem \ref{impsul}, there is a nonnegative bisymmetric $Q$ with all eigenvalues as $\lambda_0, \lambda_2,\lambda_3, \ldots, \lambda_n$ and $Q= \begin{pmatrix} A & JCJ \\ C & JAJ \end{pmatrix}$, where $A$ and $C$ are $\frac{n}{2} \times \frac{n}{2}$ matrices, $A=A^{T}, JCJ=C^{T}$. Then the matrix $\begin{pmatrix}
A &&JCJ \\&\lambda_1&\\C&&JAJ \end{pmatrix}$ would be our nonnegative bisymmetric matrix.

If $M=1, S=1$ and 1 $\in K$, by condition (1) of the assumption, it implies that $\lambda_0 + T_1 \geq 0$, and hence this case similar to the previous case.

If $M=1, S>1$ and 1 $\notin K$, by (2), it implies that $\displaystyle \lambda_0 + \sum_{i=2}^S T_i \geq 0$. Then by Theorem \ref{impsul}, there is a nonnegative bisymmetric $Q_1$ with all eigenvalues obtained from all numbers in $\lbrace\lambda_0\rbrace, \Lambda_2, \Lambda_3, \ldots, \Lambda_S$. Since $\lambda_1 + T_1 \geq 0$ and $|\Lambda_1|$ is odd, by Theorem \ref{impsul} again, there is a nonnegative bisymmetric
$\begin{pmatrix} A & JCJ \\ C & JAJ \end{pmatrix}$
with the eigenvalues obtained from $\lbrace \lambda_1 \rbrace, \Lambda_1$. Thus, the matrix
$\begin{pmatrix}
A &&JCJ \\&Q&\\C&&JAJ
\end{pmatrix}$ is our desired matrix.

If $M=1, S>1$ and 1 $\in K$, then by (2), $\displaystyle \lambda_0+\lambda_1+T_1+\sum_{i=2}^{S} T_i \geq 0.$ By Theorem \ref{impsul}, there is a nonnegative bisymmetric matrix $Q_1$ with all eigenvalues obtained from all the numbers in $\lbrace\lambda_0+\lambda_1+T_1\rbrace, \Lambda_2, \Lambda_3, \ldots, \Lambda_S.$ Since $-T_1+T_1=0$, by Theorem \ref{impsul} again, there is a nonnegative bisymmetric matrix $Q_2$ with all eigenvalues obtained from all numbers in $\lbrace\ -T_1 \rbrace$ and $\Lambda_1.$ If $\lambda_0 + T_1 \geq -(\lambda_1+T_1)$, then applying Lemma \ref{lf3} with $\varepsilon = -(\lambda_1+T_1)$, (or $\varepsilon = \lambda_0+T_1$), we get the desired matrix.

Now, let $M \geq 2$ and suppose the assertion is true for all system of $\lambda$'s satisfying the assumption of the assertion with the length less than $M$. If there is $j$ such that $1 \leq j \leq M$ and $\lambda_j+T_j \geq 0$, by Theorem \ref{impsul}, there is a nonnegative bisymmetric matrix
$\begin{pmatrix}
A & JCJ \\ C & JAJ
\end{pmatrix}$
with eigenvalues obtained from all numbers in $\lbrace \lambda_j \rbrace$ and $\Lambda_j.$ Note that the system
$$\lambda_0, \lambda_1, \ldots \lambda_{j-1},\lambda_{j+1}, \ldots, \lambda_M, T_S, \ldots, T_{j+1},T_{j-1}, \ldots,T_1$$
satisfies (1) and (2). By the induction hypothesis, there is a nonnegative bisymmetric matrix $Q$ with all eigenvalues obtained from all numbers in

$$\lbrace\lambda_0, \lambda_1, \ldots,\lambda_{j-1}, \lambda_{j+1}, \ldots, \lambda_n \rbrace, \Lambda_1,\Lambda_2, \ldots \Lambda_{j-1}, \Lambda_{j+1}, \ldots \Lambda_S.$$
\\
Therefore the matrix $\begin{pmatrix}
A & & JCJ \\
  &Q&     \\
C & & JAJ
\end{pmatrix}$ solve the problem.

Let $\lambda_i+T_i < 0$, for all $i=1,2,\ldots,M.$ Then it's easy to check that the system $\lambda_0+\lambda_1+T_1, \lambda_2, \ldots,\lambda_M, T_S, \ldots,T_2$ satisfies the assumptions (1) and (2). Therefore, by the induction hypothesis, there is a nonnegative bisymmetric matrix $Q_1$ with all eigenvalues obtained from all numbers in $\lbrace\lambda_0+\lambda_1+T_1\rbrace, \Lambda_2, \ldots \Lambda_S$ and $\lambda_0+\lambda_1+T_1 \geq \lambda_2.$ Since $-T_1+T_1=0$, by Theorem \ref{impsul} again, there is a nonnegative bisymmetric matrix $Q_2$ with all eigenvalues obtained from all numbers in $\lbrace\ -T_1 \rbrace$ and $\Lambda_1.$ If $\lambda_0 + T_1 \geq -(\lambda_1+T_1)$, then we apply Lemma \ref{lf3} with $\varepsilon = -(\lambda_1+T_1)$, (or $\varepsilon = \lambda_0+T_1$). Then we solve the problem.
\end{proof}
\end{thm}

\bigskip

\begin{thm} \label{impboro2}
Let $\lambda_0 \geq \lambda_1 \geq \ldots \geq \lambda_M \geq 0 > \lambda_{M+1} \geq \ldots \geq \lambda_n$ be real numbers. If there exists a partition $\Lambda=\Lambda_1 \cup \Lambda_2 \cup \ldots \cup \Lambda_{M-1}$ of $\lbrace\lambda_{M+1},\lambda_{M+2},\ldots,\lambda_n\rbrace$ where $|\Lambda_j|$ is odd for $j=1,2,\ldots,M-1,$ and $\displaystyle T_k =\sum_{\lambda_i \in \Lambda_k} \lambda_{i}$, $T_{M-1} \geq \ldots \geq T_1$ satisfying

\qquad (1) $\displaystyle \lambda_0 + \sum_{i\in K, i<k} (\lambda_i + T_i)+ T_k \geq 0$ \ \ \ , for all $k \in K$,

\qquad (2) $\displaystyle \lambda_0 + \sum_{i\in K} (\lambda_i + T_i)  \geq 0,$
\\where $K=\lbrace i \in \lbrace 1,2,\ldots,M-1\rbrace \mid \lambda_i + T_i < 0\rbrace$, then there is a nonnegative bisymmetric matrix with all eigenvalues as $\lambda_0,\lambda_1,\ldots,\lambda_n.$

\begin{proof}
We prove by induction on $M$.

If $M=0$ or $1$ then the assertion is clear because there is no partition $\Lambda$ of negative numbers.

If $M=2$ then $\Lambda_1$=$\lbrace\lambda_{3},\lambda_{4}, ldots,\lambda_n\rbrace$. If $1 \notin K$ then $\lambda_1 + T_1 \geq 0$ implies $\lambda_0 + T_1 \geq 0$. Similarly, if $1 \in K$ then $\lambda_0 + T_1 \geq 0$ by condition (1). Thus
$\displaystyle 0 \leq \lambda_0+T_1 = \lambda_0 + \sum_{i=3}^n \lambda_i.$ By Theorem \ref{impsul}, there is a nonnegative bisymmetric matrix $Q$ with all eigenvalues $\lambda_0, \lambda_3, \lambda_4,\ldots, \lambda_n.$ So, the matrix
$$\begin{pmatrix}
\frac{\lambda_1+\lambda_2}{2}&&\frac{\lambda_1-\lambda_2}{2} \\
&Q& \\
\frac{\lambda_1-\lambda_2}{2}&&\frac{\lambda_1+\lambda_2}{2}
\end{pmatrix}$$
is our desired matrix.

Now, let $M \geq 3$ and suppose the assertion is true for all system of $\lambda$'s  satisfying the assumption of the assertion with the length less than $M$. The proof of this step is similar to the one of Theorem \ref{impboro}, so we omit the proof.
\end{proof}
\end{thm}

\bigskip

\begin{thm} \label{impboro3}
Let $\lambda_0 \geq \lambda_1 \geq \ldots \geq \lambda_M \geq 0 > \lambda_{M+1} \geq \ldots \geq \lambda_n$ be real numbers. If there exists a partition $\Lambda=\Lambda_1 \cup \Lambda_2 \cup \ldots \cup \Lambda_S$ of $\lbrace\lambda_{M+1},\lambda_{M+2},\ldots,\lambda_n\rbrace$, with $M > S$, $|\Lambda_j|$ is odd for $j=1, 2,\ldots,S,$ and $\displaystyle T_k =\sum_{\lambda_i \in J_k} \lambda_{i}$, $T_S \geq T_{S-1} \geq \ldots \geq T_1$ satisfying

\qquad (1) $\displaystyle \lambda_0 + \sum_{i\in K, i<k} (\lambda_i + T_i)+ T_k \geq 0$, for all $k \in K$, and

\qquad (2) $\displaystyle \lambda_0 + \sum_{i\in K} (\lambda_i + T_i) \geq 0,$
\\where $K=\lbrace i \in \lbrace 1,2, \ldots,S \rbrace \mid \lambda_i + T_i < 0\rbrace$, then there is a nonnegative bisymmetric matrix with all eigenvalues as $\lambda_0,\lambda_1,\ldots,\lambda_n.$

\begin{proof}
If $M-S$ is even then we reduce the system by omitting $\lambda_{S+1},\ldots, \lambda_M.$ The new system satisfies the condition in Theorem \ref{impboro}. Then there is a nonnegative bisymmetric matrix $Q$ with all eigenvalues obtained from all numbers in $\lbrace \lambda_0, \lambda_1,\ldots, \lambda_S \rbrace$, $\Lambda_1, \ldots \Lambda_S.$ Since $\lambda_{S+1} \geq \ldots \geq \lambda_M \geq 0$, we have the matrix
$$\begin{pmatrix}
\frac{\lambda_{M-1}+\lambda_M}{2} & &&&&&\frac{\lambda_{M-1}-\lambda_M}{2}\\
&\ddots&&&&\reflectbox{$\ddots$}& \\
&&\frac{\lambda_{S+1}+\lambda_{S+2}}{2}&&\frac{\lambda_{S+1}-\lambda_{S+2}}{2}\\
&&&Q&&&&\\
&&\frac{\lambda_{S+1}-\lambda_{S+2}}{2}&&\frac{\lambda_{S+1}+\lambda_{S+2}}{2}\\
&\reflectbox{$\ddots$}&&&&\ddots&& \\
\frac{\lambda_{M-1}-\lambda_M}{2} & &&&&&\frac{\lambda_{M-1}+\lambda_M}{2}\\
\end{pmatrix}$$
is our solution.

If $M-S$ is odd then we reduce the system by omitting $\lambda_{S+2}, \ldots,\lambda_M.$ The new system satisfies the condition in Theorem \ref{impboro2}. Then there is a nonnegative bisymmetric matrix $Q$ with all eigenvalues obtained from all numbers in $\lbrace \lambda_0, \lambda_1,\ldots,$ $\lambda_{S+1} \rbrace$, $\Lambda_1, \ldots \Lambda_S.$ Since $\lambda_{S+2} \geq \ldots \geq \lambda_M \geq 0$, we have the matrix
$$\begin{pmatrix}
\frac{\lambda_{M-1}+\lambda_M}{2} &&&&&&\frac{\lambda_{M-1}-\lambda_M}{2}\\
&\ddots&&&&\reflectbox{$\ddots$}& \\
&&\frac{\lambda_{S+2}+\lambda_{S+3}}{2}&&\frac{\lambda_{S+2}-\lambda_{S+3}}{2}&&\\
&&&Q&&&\\
&&\frac{\lambda_{S+2}-\lambda_{S+3}}{2}&&\frac{\lambda_{S+2}+\lambda_{S+3}}{2}&&\\
&\reflectbox{$\ddots$}&&&&\ddots& \\
\frac{\lambda_{M-1}-\lambda_M}{2}&&&&&&\frac{\lambda_{M-1}+\lambda_M}{2}\\
\end{pmatrix}$$
is our solution. This completes the proof.
\end{proof}
\end{thm}
\bigskip

\begin{thm} \label{boro4}
Let $\lambda_0 \geq \lambda_1 \geq \ldots \geq \lambda_M \geq 0 > \lambda_{M+1} \geq \ldots \geq \lambda_n$ be real numbers. If there exists a partition $\Lambda=\Lambda_1 \cup \Lambda_2 \cup \ldots \cup \Lambda_S$ of $\lbrace\lambda_{M+1},\lambda_{M+2},\ldots,\lambda_n\rbrace$ with $|\Lambda_j|$ is odd for $j=1,2,\ldots, \emph{min} \lbrace M, S \rbrace,$ and $\displaystyle T_k =\sum_{\lambda_i \in \Lambda_k} \lambda_{i}$, $T_S \geq T_{S-1} \geq \ldots \geq T_1$ satisfying

\qquad (1) $\displaystyle \lambda_0 + \sum_{i\in K, i<k} (\lambda_i + T_i)+ T_k \geq 0$, for all $k \in K$, and

\qquad (2) $\displaystyle \lambda_0 + \sum_{i\in K} (\lambda_i + T_i)+ \sum_{j=M+1}^S T_j  \geq 0,$
\\where $K=\lbrace i \in \lbrace 1,2,\ldots, \emph{min} \lbrace M,S \rbrace \rbrace \mid \lambda_i + T_i < 0\rbrace$, then there is a nonnegative bisymmetric matrix with all eigenvalues as $\lambda_0,\lambda_1,\ldots,\lambda_n.$
\end{thm}

\begin{proof}
The theorem follows immediately from Theorem \ref{impboro} and Theorem \ref{impboro3} .
\end{proof}

\medskip

The statement of the next theorem is adapted from Theorem \ref{soto4} and it is a variance of Theorem \ref{julio1} and Theorem \ref{julio2}. The sufficient condition in the next theorem and the sufficient condition in Theorem \ref{julio1} and Theorem \ref{julio2} are very similar but they are different in terms of the requirements of the nonnegative matrices in condition (1) and (2).

\bigskip

\begin{thm} \label{imsoto3} 
Let $\lambda_1 \geq \lambda_1 \geq \ldots \geq \lambda_n$ be real numbers and let $\omega_1, \ldots, \omega_S$ be nonnegative numbers, where $S \leq n$ and $ 0 \leq \omega_k \leq \lambda_1, i=1, \ldots, S.$ Suppose that

\qquad (1) there is a partition $\Lambda_1 \cup \ldots \cup \Lambda_S$ of $\lbrace \lambda_1, \ldots, \lambda_n \rbrace$, in which there is a $\Lambda_i$ of odd size at most one set,  $\Lambda_j=\lbrace \lambda_{j1}, \lambda_{j2} \ldots \lambda_{jp_j} \rbrace$, $\lambda_{jk} \geq \lambda_{j(k+1)}$, $\lambda_{j1} \geq 0$ , and $\lambda_{11}=\lambda_1$, such that for each $j=1, \ldots, S$, the set $\Gamma_j= \lbrace \omega_j, \lambda_{j2}, \ldots, \lambda_{jp_j} \rbrace$ is realizable by a nonnegative bisymmetric matrix with the Perron root $\omega_j$, and

\qquad (2) there is an $S \times S$ nonnegative symmetric matrix $B$ with all eigenvalues as $\lambda_{11}, \lambda_{21}, \ldots, \lambda_{S1}$ and diagonal entries $\omega_1, \omega_2, \ldots, \omega_S.$ \\
Then $\lbrace \lambda_1, \lambda_2, \ldots, \lambda_n \rbrace$ is realizable by nonnegative bisymmetric matrix.

\begin{proof}
First, we consider in the case that $\Lambda_j$ is of even size for all $j=1, 2, \ldots, S$. Let for each $j$, $\Gamma_j$ is realizable by a nonnegative bisymmetric matrix $Q_j$ of size even. By Theorem \ref{canthm1}, we can write $Q_j=\begin{pmatrix}
A_j & JC_jJ \\
C_j & JA_jJ
\end{pmatrix}$
, where $A_j$ and $C_j$ are $\frac{p_j}{2} \times \frac{p_j}{2}$ matrices, $A_j=A_j^T$ and $C_j^T=JC_jJ$. Then $$\widehat{Q} = \begin{pmatrix}
A_S    &                       &      &       &                      &  JC_SJ\\
       & \ddots                &      &       &\reflectbox{$\ddots$} &       \\
       &                       &  A_1 & JC_1J &                      &       \\
       &                       &  C_1 & JA_1J &                      &       \\
       & \reflectbox{$\ddots$} &      &       &\ddots                &       \\
C_S    &                       &      &       &                      &  JA_SJ\\
\end{pmatrix} \qquad (\ast\ast)$$
is a nonnegative bisymmetric matrix with all eigenvalues obtained from all numbers in $\Gamma_j$, for $j=1, \ldots, S$.
By Lemma \ref{canlemma2}, for each $j=1, \ldots, S$ we can find the unit nonngative eigenvector of $Q_j$ corresponding to $\omega_j$ in the form $\begin{pmatrix}
v_j \\ Jv_j
\end{pmatrix}$.
Then
$$x_1 = \begin{pmatrix} 0\\ \vdots \\0 \\v_1 \\Jv_1 \\0 \\ \vdots\\0 \end{pmatrix},
 x_2 = \begin{pmatrix} 0\\ \vdots \\v_2 \\0 \\ 0 \\Jv_2 \\ \vdots\\0 \end{pmatrix},\ldots,
 x_S =\begin{pmatrix} v_S\\ 0 \\ \vdots \\0 \\ 0 \\ \vdots \\ 0 \\ Jv_S \end{pmatrix}$$
form an orthonormal set of eigenvectors of $\widehat{Q}$ corresponding to $\omega_1, \ldots, \omega_S$, respectively. Obviously, the $\left(\displaystyle  \sum_{j=1}^S p_j\right) \times S$ matrix $X=\begin{pmatrix}
x_S & x_{S-1}& \cdots& x_1
\end{pmatrix}$ is a nonnegative matrix with $JX=X$, where $J$ is the $\displaystyle \left(\sum_{j=1}^S p_j\right) \times \left(\sum_{j=1}^S p_j\right) $ reverse identity matrix, and $\widehat{Q}X=X\Omega$, where $\Omega$ is diag $(\omega_S, \omega_{S-1}, \ldots, \omega_1)$. Therefore $\widehat{Q}+X(B-\Omega)X^T$ is a nonnegative matrix with eigenvalues $\lambda_1, \lambda_2, \ldots, \lambda_n$ by Theorem \ref{sotos}. Moreover, $\widehat{Q}+X(B-\Omega)X^T$ is a bisymmetric matrix because $\widehat{Q}$ is a bisymmetric and $JX=X$.

If there is a $p$ such that $0 \leq p \leq S$, $\vert \Lambda_p \vert$ is odd number and $\Gamma_p$ is realizable by nonnegative bisymmetric matrix $Q_p$
then we set $Q_p$ in the center of the matrix $\widehat{Q}$ in ($\ast \ast$) and the construction follows from the previous case.
\end{proof}
\end{thm}

\begin{exa} 
We construct a bisymmetric nonnegative matrix with eigenvalues $9, 2, -1, -2, -3, -4$. We take the partion $\Lambda_1 = \lbrace -2, -3, -4 \rbrace$ and $\Lambda_2=\lbrace -1 \rbrace$ of $\lbrace -1, -2, -3, -4 \rbrace$. Then it satisfies the condition in the Theorem \ref{boro4}. In fact, by Theorem \ref{impsul}, the set $\lbrace 9, -2, -3, -4 \rbrace$ is realizable by the nonnegative bisymmetric matrix
$$A_1=\begin{pmatrix}
0&\sqrt{\frac{15}{2}}&\sqrt{\frac{15}{2}}&3\\
\sqrt{\frac{15}{2}}&0&4&\sqrt{\frac{15}{2}}\\
\sqrt{\frac{15}{2}}&4&0&\sqrt{\frac{15}{2}}\\
3&\sqrt{\frac{15}{2}}&\sqrt{\frac{15}{2}}&0
\end{pmatrix}.$$
Also, the set $\lbrace 2, -1 \rbrace$ is realizable by the nonnegative bisymmetric matrix \\
$\begin{pmatrix}
0.5&1.5\\
1.5&0.5
\end{pmatrix}$.
Then the matrix
$$\begin{pmatrix}
0.5&&&&&1.5\\
&0&\sqrt{\frac{15}{2}}&\sqrt{\frac{15}{2}}&3&\\
&\sqrt{\frac{15}{2}}&0&4&\sqrt{\frac{15}{2}}&\\
&\sqrt{\frac{15}{2}}&4&0&\sqrt{\frac{15}{2}}&\\
&3&\sqrt{\frac{15}{2}}&\sqrt{\frac{15}{2}}&0&\\
1.5&&&&&0.5
\end{pmatrix}$$
is our desired matrix.
\end{exa}
\bigskip

\begin{exa} 
We construct a bisymmetric nonnegative matrix with eigenvalues $9, 5, 1, 1, -4, -4, -8$. Note that this list is not satisfy the condition in Theorem \ref{boro4}. However, it satisfies condition in Theorem \ref{imsoto3} with partition $\Lambda_1=\lbrace 9, -8 \rbrace$ and $\Lambda_2=\lbrace 5, 1, 1, -4, -4 \rbrace.$ In fact, the set $\Gamma_1=\lbrace 8, -8 \rbrace$ and $\Gamma_2=\lbrace 6, 1, 1, -4, -4 \rbrace$ are realizable by the nonnegative bisymmetric matrix
$$A_1=\begin{pmatrix}
0&8\\
8&0
\end{pmatrix}
\text{ and } A_2=\begin{pmatrix}
0&\frac{3+\sqrt{5}}{2}&\frac{3-\sqrt{5}}{2}&\frac{3-\sqrt{5}}{2}&\frac{3+\sqrt{5}}{2}\\
\frac{3+\sqrt{5}}{2}&0&\frac{3+\sqrt{5}}{2}&\frac{3-\sqrt{5}}{2}&\frac{3-\sqrt{5}}{2}\\
\frac{3-\sqrt{5}}{2}&\frac{3+\sqrt{5}}{2}&0&\frac{3+\sqrt{5}}{2}&\frac{3-\sqrt{5}}{2}\\
\frac{3-\sqrt{5}}{2}&\frac{3-\sqrt{5}}{2}&\frac{3+\sqrt{5}}{2}&0&\frac{3+\sqrt{5}}{2}\\
\frac{3+\sqrt{5}}{2}&\frac{3-\sqrt{5}}{2}&\frac{3-\sqrt{5}}{2}&\frac{3+\sqrt{5}}{2}&0
\end{pmatrix},$$
respectively( see \cite{soto2} for the construction of $A_2$). Then

$$\widehat{Q} =\begin{pmatrix}
0&&8\\
&A_2&\\
8&&0
\end{pmatrix}, \Omega=\begin{pmatrix}
8&0\\
0&6
\end{pmatrix} \text{ and }
X=\begin{pmatrix}
0&\frac{\sqrt{2}}{2}\\
\frac{\sqrt{5}}{5}&0\\
\frac{\sqrt{5}}{5}&0\\
\frac{\sqrt{5}}{5}&0\\
\frac{\sqrt{5}}{5}&0\\
\frac{\sqrt{5}}{5}&0\\
0&\frac{\sqrt{2}}{2}
\end{pmatrix}.$$

Now, we find the nonnegative symmetric matrix with eigenvalues $9, 5$ and with diagonal entries $8, 6$. Then the matrix $B=\begin{pmatrix}
8&2\\2&6
\end{pmatrix}$ is required.

Therefore
$$\widehat{Q}+X(B-\Omega) X^T=\begin{pmatrix}
0&\frac{\sqrt{10}}{5}&\frac{\sqrt{10}}{5}&\frac{\sqrt{10}}{5}&\frac{\sqrt{10}}{5}&\frac{\sqrt{10}}{5}&8\\
\frac{\sqrt{10}}{5}&0&\frac{3+\sqrt{5}}{2}&\frac{3-\sqrt{5}}{2}&\frac{3-\sqrt{5}}{2}&\frac{3+\sqrt{5}}{2}&\frac{\sqrt{10}}{5}\\
\frac{\sqrt{10}}{5}&\frac{3+\sqrt{5}}{2}&0&\frac{3+\sqrt{5}}{2}&\frac{3-\sqrt{5}}{2}&\frac{3-\sqrt{5}}{2}&\frac{\sqrt{10}}{5}\\
\frac{\sqrt{10}}{5}&\frac{3-\sqrt{5}}{2}&\frac{3+\sqrt{5}}{2}&0&\frac{3+\sqrt{5}}{2}&\frac{3-\sqrt{5}}{2}&\frac{\sqrt{10}}{5}\\
\frac{\sqrt{10}}{5}&\frac{3-\sqrt{5}}{2}&\frac{3-\sqrt{5}}{2}&\frac{3+\sqrt{5}}{2}&0&\frac{3+\sqrt{5}}{2}&\frac{\sqrt{10}}{5}\\
\frac{\sqrt{10}}{5}&\frac{3+\sqrt{5}}{2}&\frac{3-\sqrt{5}}{2}&\frac{3-\sqrt{5}}{2}&\frac{3+\sqrt{5}}{2}&0&\frac{\sqrt{10}}{5}\\
8&\frac{\sqrt{10}}{5}&\frac{\sqrt{10}}{5}&\frac{\sqrt{10}}{5}&\frac{\sqrt{10}}{5}&\frac{\sqrt{10}}{5}&0
\end{pmatrix}$$
is our desired matrix.
\end{exa}
\bigskip

\section{Sufficient Conditions for the Bisymmetric Positive Eigenvalue Problem}

In Theorem 3.2 of \cite{fdl}, Fiedler showed that, if $A$ is a nonnegative symmetric matrix with all eigenvalues as $\lambda_0 \geq \lambda_1 \geq \ldots \geq \lambda_n$ and $\varepsilon > 0$ is given then there is a positive symmetric matrix $B$ with all eigenvalues as $\lambda_0+\varepsilon, \lambda_1,\ldots, \lambda_n.$ Moreover, in the proof of that theorem, we can find the positive symmetric matrix $R$ such that $B=A+R$. Therefore, we can modify Theorem 3.2 of \cite{fdl} to the following result.
\bigskip

\begin{thm} (Fiedler, \cite{fdl}, 1974) \label{fdl3}
If $A$ is a nonnegative symmetric matrix with all eigenvalues as $\lambda_0 \geq \lambda_1 \geq \ldots \geq \lambda_n$ and $\varepsilon > 0$ is given then there is a positive symmetric matrix $R$ such that $A+R$ has all eigenvalues as $\lambda_0+\varepsilon, \lambda_1, \ldots, \lambda_n.$
\end{thm}
\bigskip

Next, we improve Theorem \ref{fdl3} to bisymmetric case.

\begin{thm} \label{pbs} 
If $Q$ is a nonnegative bisymmetric matrix with all eigenvalues as $\lambda_0, \lambda_1 \geq \ldots \geq \lambda_n$ and $\varepsilon > 0$ then there is a positive bisymmetric matrix $\hat{Q}$ with all eigenvalues as $\lambda_0+\varepsilon, \lambda_1,\ldots, \lambda_n.$
\begin{proof}
Let $Q$ be a nonnegative bisymmetric matrix with all eigenvalues as $\lambda_0, \lambda_1,$ $\ldots,$ $\lambda_n.$

\textit{Case 1}: Suppose $n$ is odd. Then we can write $Q=\begin{pmatrix}
A & JCJ \\
C & JAJ
\end{pmatrix}$
, where $A$ and $C$ are $\frac{n+1}{2} \times \frac{n+1}{2}$ matrices, $A=A^T$ and $C^T=JCJ.$ By Theorem \ref{canthm2} and Lemma \ref{canlemma1}, all eigenvalues of $Q$ are obtained from the eigenvalues of $A+JC$ and $A-JC$, and $\lambda_0$ is an eigenvalue of $A+JC.$ If $\lambda_0, \lambda_{1_1}, \lambda_{1_2},\ldots, \lambda_{1_{\frac{n-1}{2}}}$ are all eigenvalues of $A+JC$, by Theorem \ref{fdl3}, there is a positive symmetric matrix $R$ such that $A+JC+R$ has all eigenvalues as $\lambda_0+\varepsilon, \lambda_{1_1}, \lambda_{1_2},..., \lambda_{1_{\frac{n-1}{2}}}.$ So, the matrix
$$\begin{pmatrix}
A+\frac{1}{2}R & JCJ+\frac{1}{2}RJ \\
C+\frac{1}{2}JR & JAJ+\frac{1}{2}JRJ
\end{pmatrix}$$ is the one we are looking for.

\textit{Case 2}: Suppose $n$ is even. Then we can write $Q=\begin{pmatrix}
A   & x  & JCJ \\
x^T & p  & x^TJ \\
C   & Jx & JAJ
\end{pmatrix}$
, where $A$ and $C$ are $\frac{n}{2} \times \frac{n}{2}$ matrices, $A=A^T$ and $C^T=JCJ$. By Theorem \ref{canthm2} and Lemma \ref{canlemma1}, all eigenvalues of $Q$ obtained from eigenvalues of
$\begin{pmatrix}
p & \sqrt{2}x^T \\
\sqrt{2}x & A+JC
\end{pmatrix}$
and $A-JC$ and $\lambda_0$ is an eigenvalue of $\begin{pmatrix}
p & \sqrt{2}x^T \\
\sqrt{2}x & A+JC
\end{pmatrix}.$ If $\lambda_0, \lambda_{1_1}, \lambda_{1_2},\ldots, \lambda_{1_{\frac{n}{2}}}$ are all eigenvalues of
$\begin{pmatrix}
p & \sqrt{2}x^T \\
\sqrt{2}x & A+JC
\end{pmatrix}$, by Theorem \ref{fdl3}, there is a positive symmetric
$R=\begin{pmatrix}
c & y^T \\
y & R_1
\end{pmatrix}$, where $R_1$ is an $\frac{n}{2} \times \frac{n}{2}$ positive symmetric matrix, and the matrix
$\begin{pmatrix}
p+c         & \sqrt{2}x^T + y^T \\
\sqrt{2}x+y & A+JC+R_1
\end{pmatrix}$ has all eigenvalues as $\lambda_0+\varepsilon, \lambda_{1_1}, \lambda_{1_2}, \ldots, \lambda_{1_{\frac{n}{2}}}.$ So, the matrix
$$\begin{pmatrix}
A+\frac{1}{2}R_1           & x+\frac{1}{\sqrt{2}}y   &JCJ+\frac{1}{2}RJ_1 \\
x^T+\frac{1}{\sqrt{2}}y^T  & p+c                     &(x^T+\frac{1}{\sqrt{2}}y^T)J \\
C+\frac{1}{2}JR_1          &J(x+\frac{1}{\sqrt{2}}y) &JAJ+\frac{1}{2}JR_1J
\end{pmatrix}$$ is our desired one.
\end{proof}
\end{thm}

\bigskip

\begin{cor}
If $Q$ is a nonnegative bisymmetric matrix with all eigenvalues as $\lambda_0 \geq \lambda_1 \geq \ldots \geq \lambda_n$ and $\varepsilon > 0$ then there is positive bisymmetric matrix $P$ such that $Q+P$ has all eigenvalues as $\lambda_0+\varepsilon, \lambda_1,\ldots, \lambda_n.$
\begin{proof}
This result follows from the construction in the proof of Theorem \ref{pbs}.
\end{proof}
\end{cor}
\bigskip

\begin{thm} \label{pboro} 
Let $\lambda_0 \geq \lambda_1 \geq \ldots \geq \lambda_M \geq 0 > \lambda_{M+1} \geq \ldots \geq \lambda_n$ be real numbers. If there exists a partition $\Lambda=\Lambda_1 \cup \Lambda_2 \cup \ldots \cup \Lambda_S$ of $\lbrace\lambda_{M+1},\lambda_{M+2},\ldots,\lambda_n\rbrace$ with $|\Lambda_j|$ is odd for $j=1,2,\ldots, \emph{min} \lbrace M, S\rbrace,$ and $\displaystyle T_k =\sum_{\lambda_i \in \Lambda_k} \lambda_{i}$, $T_S \geq T_{S-1} \geq \ldots \geq T_1$ satisfying

\qquad (1) $\displaystyle \lambda_0 + \sum_{i\in K, i<k} (\lambda_i + T_i)+ T_k > 0$, \qquad \quad for all $k \in K$, and

\qquad (2) $\displaystyle \lambda_0 + \sum_{i\in K} (\lambda_i + T_i)+ \sum_{j=M+1}^S T_j  > 0,$
\\where $K=\lbrace i \in \lbrace 1,2,\ldots, \emph{min} \lbrace M,S \rbrace \rbrace \mid \lambda_i + T_i < 0\rbrace$, then there is a positive bisymmetric matrix with all eigenvalues as $\lambda_0,\lambda_1,...,\lambda_n.$
\begin{proof}
Let

$$\varepsilon = \text{min} \left\lbrace \min_{k \in K} \lbrace \lambda_0 + \sum_{i\in K, i<k} (\lambda_i + T_i)+ T_k \rbrace , \lambda_0 + \sum_{i\in K} (\lambda_i + T_i)+ \sum_{j=M+1}^S T_j \right\rbrace.$$

Then the system $\lambda_0 - \varepsilon, \lambda_1,\ldots, \lambda_n$ satisfies the condition in Theorem \ref{boro4} . Thus there is an  $(n+1) \times (n+1)$ nonnegative bisymmetric matrix $Q$ with all eigenvalues as $\lambda_0 - \varepsilon, \lambda_1, \ldots, \lambda_n.$ By applying Theorem \ref{pbs}, the proof is complete.
\end{proof}
\end{thm}
\bigskip

\section{The BNIEP with the Prescribe Diagonal Entries}

\bigskip

Note that the condition (2) in Theorem \ref{julio1} and Theorem \ref{julio2} requires an existence of a nonnegative bisymmetric matrix with the prescribed eigenvalues and with the prescribed diagonal entries. So in this section we provide a sufficient condition for the BNIEP with the given diagonal entries. We begin this section by proving the following lemma.

\bigskip

\begin{lem} \label{bidia1}
Let $\alpha_0 \geq \alpha_1 \geq \alpha_2$ be real numbers and $a_0, a_1$ be nonnegative real numbers. Then there is a $3 \times 3$ nonnegative bisymmetric matrix with all eignenvalues as $\alpha_0, \alpha_1, \alpha_2$ and $a_1, a_0, a_1$ are in the main diagonal entries if and only if there is a $j \in \lbrace 1, 2\rbrace$ such that
$$a_1 \geq \alpha_j$$
$$\alpha_0+\alpha_j \geq 2a_1$$
$$\alpha_1+\alpha_2 \leq 2a_1$$
$$\alpha_0+\alpha_1+\alpha_2 = a_0+2a_1$$
\begin{proof}

($\Leftarrow$) It's easy to check that the matrix
$$\begin{pmatrix}
a_1 & \rho & a_1-\alpha_j \\
\rho&a_0&\rho\\
a_1-\alpha_j & \rho &a_1
\end{pmatrix},$$
where $\rho=\sqrt{\dfrac{(\alpha_0-a_0)(\alpha_0+\alpha_j-2a_1)}{2}}$,
is a $3 \times 3$ nonnegative bisymmetric matrix with all eigenvalues as $\alpha_0, \alpha_1, \alpha_2.$

($\Rightarrow$) Let
$Q=\begin{pmatrix}
a_1 & \rho & \xi \\
\rho& a_0  & \rho \\
\xi & \rho & a_1
\end{pmatrix}$
have all eigenvalues as $\alpha_0, \alpha_1, \alpha_2.$ Since the sum of all eigenvalues is equal to the trace of the matrix, we have $\alpha_0+\alpha_1+\alpha_2 = a_0+2a_1.$ Moreover, by Theorem \ref{canthm2}, $Q$ is orthogonally similar to the matrix
$$\begin{pmatrix}
a_1-\xi&              &               \\
       & a_0          &  \sqrt{2}\rho \\
       & \sqrt{2}\rho &  a_1+\xi
\end{pmatrix}.$$
By Lemma \ref{canlemma1}, $\alpha_0$ is an eigenvalue of $\begin{pmatrix}
a_0 & \sqrt{2}\rho \\
\sqrt{2}\rho & a_1+\xi
\end{pmatrix}.$ Then $a_1-\xi$ must be $\alpha_1$ or $\alpha_2$.

\textit{Case I}: $a_1-\xi=\alpha_1$. Then $a_1-\alpha_1=\xi \geq 0$ and it implies that $a_1 \geq \alpha_1$ and $2a_1 \geq \alpha_1+\alpha_2.$ Since the matrix $\begin{pmatrix}
a_0 & \sqrt{2}\rho \\
\sqrt{2}\rho & a_1+\xi
\end{pmatrix}$ has eigenvalues $\alpha_0, \alpha_2$, its characteristic polynomial  $x^2-(a_0+a_1+\xi)x+a_0a_1+a_0\xi-2\rho^2$=$x^2-(a_0+2a_1-\alpha_1)x+2a_0a_1-a_0\alpha_1-2\rho^2$ is equal to $x^2-(\alpha_0+\alpha_2)+\alpha_0\alpha_2.$ So,

\qquad \qquad \quad 0  $\leq$ $2\rho^2 = 2a_0a_1-a_0\alpha_1-\alpha_0\alpha_2$

\qquad \qquad \qquad \qquad \quad $=(\alpha_0+\alpha_1+\alpha_2-a_0)a_0-a_0\alpha_1-\alpha_0\alpha_2$

\qquad \qquad \qquad \qquad \quad $=\alpha_0a_0+\alpha_1a_0+\alpha_2a_0-a_0^2-a_0\alpha_1-\alpha_0\alpha_2$

\qquad \qquad \qquad \qquad \quad $=\alpha_0a_0+\alpha_2a_0-a_0^2-\alpha_0\alpha_2$

\qquad \qquad \qquad \qquad \quad $=(\alpha_0-a_0)(a_0-\alpha_2).$

Since $2a_1 \geq \alpha_1+\alpha_2$ and $a_0+2a_1=\alpha_0+\alpha_1+\alpha_2$, we have $\alpha_0-a_0 \geq 0$. This implies that $a_0 \geq \alpha_2$. This shows that $j=1$ satisfies our conditions.

\textit{Case II}: $a_1-\xi=\alpha_2$. By the similar argument, we have $j=2$ satisfies our conditions.
\end{proof}
\end{lem}

\medskip

The next two theorems are the main results in this section.

\medskip

\begin{thm} \label{bidiagonalodd}
If $\lambda_0 \geq \lambda_1 \geq \ldots \geq \lambda_{2m}$ are real numbers and $a_0, a_1,\ldots a_m$ are nonnegative real numbers such that they satisfies all of these conditions;

(1) For each $k=1,2, \ldots m$ there is $j_k \in \lbrace 2k-1, 2k \rbrace$ such that $a_k \geq \lambda_{j_k}$ and

\qquad $\displaystyle \sum_{i=0}^{2(k-1)} \lambda_i+\lambda_{j_k} \geq 2\sum_{i=1}^k a_i,$

(2) $2a_k \geq \lambda_{2k-1} + \lambda_{2k}$, for each $k=1,2, \ldots, m$, and

(3) $\displaystyle \sum_{i=0}^{2m} \lambda_i= a_0+2\sum_{i=1}^m a_i,$ \\
then there is a nonnegative bisymmetric matrix with all eigenvalues as $\lambda_0, \ldots, \lambda_{2m}$ and with the diagonal entries as $a_m, a_{m-1}, \ldots, a_1, a_0, a_1, \ldots, a_{m-1}, a_m.$

\begin{proof} We prove this theorem by induction on $m.$ If $m=1$ then we have
$a_1 \geq \lambda_{j_1} \text{ and } \lambda_0 + \lambda_{j_1} \geq 2a_1 \text{ for some } j_1 \in \lbrace 1, 2\rbrace$. Moreover, $2a_1 \geq \lambda_1+\lambda_2$ and
$\lambda_0+\lambda_1+\lambda_2=a_0+2a_1$. So, all conditions in lemma \ref{bidia1} hold.

Next, let $m \geq 2$ and suppose the assertion is true for all systems such that $M < m$. Define

\qquad \qquad \qquad \qquad \qquad \qquad $\lambda_0^\prime \quad = \quad \lambda_0+\lambda_1+\lambda_2-2a_1,$

\qquad \qquad \qquad \qquad \qquad \qquad $\lambda_i^\prime \quad = \quad \lambda_{i+2}, \qquad \qquad \qquad i = 1,2, \ldots, 2m-2,$

\qquad \qquad \qquad \qquad \qquad \qquad $a_0^\prime \quad = \quad a_0,$

and \qquad \qquad \qquad \qquad \qquad $a_i^\prime \quad = \quad a_{i+1},$
\qquad \qquad \qquad $i=1, \ldots, m-1.$

Now, we show that $\lambda_0^\prime, \ldots, \lambda_{2(m-1)}^\prime, a_0^\prime, \ldots, a_{m-1}^\prime$ satisfy all conditions of the assumption.

(i) For $k=1, \ldots, m-1,$ set $j_k^\prime = j_{k+1}-2$. Obviously, $j_k^\prime \in \lbrace 2k-1, 2k \rbrace.$ Moreover,
$$a_k^\prime = a_{k+1} \geq \lambda_{j_{k+1}}=\lambda_{j_{k+1}+2-2}=\lambda_{j_k^\prime+2}=\lambda_{j_k^\prime}^\prime,$$
and

\qquad \qquad \quad \ \ $\displaystyle \sum_{i=0}^{2(k-1)} \lambda_i^\prime+\lambda_{j_{k}^\prime}^\prime = \lambda_0^\prime+\sum_{i=1}^{2(k-1)} \lambda_i^\prime + \lambda_{j_{k}^\prime}^\prime$

\qquad \qquad \qquad \qquad \qquad \qquad $\displaystyle=\lambda_0+\lambda_1+\lambda_2-2a_1+\sum_{i=1}^{2(k-1)} \lambda_{i+2} + \lambda_{j_{k}^\prime+2}$

\qquad \qquad \qquad \qquad \qquad \qquad  $\displaystyle =\sum_{i=0}^{2k} \lambda_i -2a_1+\lambda_{j_{k+1}}$

\qquad \qquad \qquad \qquad \qquad \qquad  $\displaystyle \geq 2\sum_{i=1}^{k+1} a_i-2a_1$

\qquad \qquad \qquad \qquad \qquad \qquad  $\displaystyle =2\sum_{i=2}^{k+1} a_i$

\qquad \qquad \qquad \qquad \qquad \qquad  $\displaystyle =2\sum_{i=1}^{k} a_i^\prime.$

(ii) For $k=1, \ldots, m-1$ we have

\qquad \qquad \qquad \qquad \qquad $2a_k^\prime=2a_{k+1} \geq \lambda_{2k+1}+\lambda_{2k+2} = \lambda_{2k-1}^\prime+\lambda_{2k}^\prime.$

(iii) Obviously, $\displaystyle \sum_{i=0}^{2(m-1)} \lambda_i^\prime = a_0^\prime + 2\sum_{i=1}^{m-1} a_i^\prime.$

Then, by the induction hypothesis, there is a nonnegative bisymmetric matrix $A$ with all eigenvalues as $\lambda_0^\prime, \ldots, \lambda_{2m-2}^\prime$ and its diagonal entries are $a_{m-1}^\prime, \ldots,$ $a_1^\prime,$ $a_0^\prime, a_1^\prime,$ $\ldots,$ $a_{m-1}^\prime.$

Now, we show that $\lambda_0^\prime \geq \lambda_1^\prime$.
If $\lambda_{j_2}=\lambda_3$ then $a_2 \geq \lambda_3$ and $\lambda_0+\lambda_1+\lambda_2+\lambda_3 \geq 2a_1+2a_2.$
So, $\lambda_0+\lambda_1+\lambda_2-2a_1-\lambda_3 \geq 2a_2-2\lambda_3 \geq 0.$
If $\lambda_{j_2} = \lambda_4$ then $a_2 \geq \lambda_4$ and $\lambda_0+\lambda_1+\lambda_2+\lambda_4 \geq 2a_1+2a_2.$
So, $\lambda_0+\lambda_1+\lambda_2-2a_1-\lambda_3 \geq 2a_2-\lambda_3-\lambda_4 \geq 0$, by the condition $(2)$.

Finally, let $$\widehat{Q} = \begin{pmatrix}
a_1                 & \rho             & a_1-\lambda_{j_1}\\
\rho                & \lambda_0^\prime & \rho \\
a_1-\lambda_{j_1}   &   \rho           &a_1
\end{pmatrix},$$ where $\rho=\sqrt{\frac{(2a_1-\lambda_1-\lambda_2)(\lambda_0+\lambda_{j_1}-2a_1)}{2}}$. Then it is easy to see that $\widehat{Q}$ has all eigenvalues as $\lambda_0, \lambda_1, \lambda_2$. By Lemma \ref{lep}, show that the matrix
$$\begin{pmatrix}
a_1&\rho u & a_1-\lambda_{j_1} \\
\rho u^T&A&\rho u^T \\
a_1-\lambda_{j_1}&\rho u&a_1
\end{pmatrix},$$
where $u$ is a unit eigenvector of $A$ corresponding to $\lambda_0^\prime$ in which $Ju=u$, has all eigenvalues as $\lambda_0, \lambda_1, \ldots, \lambda_{2m}$. The sequence of diagonal entries can be permuted by some permutation matrix. This completes the proof.
\end{proof}
\end{thm}

\bigskip

\begin{thm}
If $\lambda_0 \geq \lambda_1 \geq \ldots \geq \lambda_{2m+1}$ and $a_1 \geq \ldots \geq a_m$ and $a_0$ are nonnegative real numbers satisfy all of these conditions;

(1) $a_i \geq \lambda_{2i-1}$ for $i=1, \ldots, m$,

(2) $\displaystyle \sum_{i=0}^k \lambda_i \geq \sum_{i=1}^{ \lceil \frac{k+1}{2} \rceil} a_i+ \sum_{i=1}^{\lfloor \frac{k+1}{2} \rfloor}$ $a_i,$ for $k=1, \ldots, 2m-1$, and

(3) $\displaystyle \sum_{i=0}^{2m+1} \lambda_i=2\sum_{i=0}^m a_i$, \\
then there is a nonnegative bisymmetric matrix with all eigenvalues $\lambda_0, \ldots, \lambda_{2m+1}$ and with the diagonal entries as $a_m, a_{m-1},$ $\ldots,$ $a_0,$ $a_0,$ $a_1,$ $\ldots,$ $a_m.$

\begin{proof}
We prove the theorem by induction on $m.$ If $m=1$ then it is easy to check that the matrix
$$\begin{pmatrix}
a_1     &   \rho            & \rho          & a_1-\lambda_1 \\
\rho    &    a_0            & a_0-\lambda_3 & \rho          \\
\rho    &    a_0-\lambda_3  & a_0           & \rho \\
a_1-\lambda_1 & \rho        & \rho          & a_1
\end{pmatrix},$$
where $\rho=\dfrac{\sqrt{(\lambda_0-2a_0+\lambda_3)(\lambda_0+\lambda_1-2a_1)}}{2}$ is our solution. The induction step is similar to the one in the Theorem \ref{bidiagonalodd}.
\end{proof}
\end{thm}

\bigskip

\bibliographystyle{amsplain}
\bibliography{BNIEP}

\end{document}